\documentclass[reqno, 11pt]{amsart}
\pdfoutput=1 

\usepackage[usenames, dvipsnames]{xcolor}

\usepackage{amsmath,amssymb,amsthm,mathtools}
\usepackage{tikz}
\usetikzlibrary{automata, positioning,bbox}
\usepackage[T1]{fontenc}
\usepackage[utf8]{inputenc}
\usepackage{enumerate}
\usepackage[letterpaper, margin=1in]{geometry}
\usepackage[noBBpl]{mathpazo} 

\usepackage[longnamesfirst,numbers,sort&compress]{natbib}
\makeatletter
\def\NAT@spacechar{~}
\makeatother

\usepackage{thmtools}
\usepackage{thm-restate}
\usepackage{hyperref}
\hypersetup{
colorlinks,
linkcolor={cDeepBlue},
citecolor={green!60!black},
urlcolor={blue!80!black},
pdftitle = {Aharoni's rainbow cycle conjecture holds up to an additive constant},
pdfauthor = {}
}
\usepackage[compress, capitalise, nameinlink, noabbrev]{cleveref} 

\declaretheorem[name = Theorem, numberwithin = section, style = plain, refname = {Theorem,Theorems}, Refname = {Theorem,Theorems}]{theorem}
\declaretheorem[name = Claim, numberlike = theorem, style = plain, refname = {Claim,Claims}, Refname = {Claim,Claims}]{claim}

\declaretheorem[name = Corollary, numberlike = theorem, style = plain, refname = {Corollary,Corollaries}, Refname = {Corollary,Corollaries}]{corollary}

\declaretheorem[name = Lemma, numberlike = theorem, style = plain, refname = {Lemma,Lemmas}, Refname = {Lemma,Lemmas}]{lemma}

\declaretheorem[name = Conjecture, numberlike = theorem, style = plain, refname = {Conjecture,Conjectures}, Refname = {Conjecture, Conjetures}]{conjecture}

\newcommand{\boundellipse}[3]
{(#1) ellipse (#2 and #3)
}

\newcommand{\CH}{Caccetta-Häggkvist }

\newcommand{\R}{\mathbb{R}}

\newcommand{\defn}[1]{\textcolor{cMaroon}{\emph{#1}}}
\newcommand{\df}{\mathsf{def}}
\renewcommand{\geq}{\geqslant}
\renewcommand{\leq}{\leqslant}

\DeclarePairedDelimiter{\abs}{\lvert}{\rvert}

\definecolor{cMaroon}{HTML}{93152a}

\definecolor{cYellow}{HTML}{f0e442}
\definecolor{cOrange}{HTML}{f7903b}
\definecolor{cRed}{HTML}{ff3341}
\definecolor{cPurple}{HTML}{da58c2}
\definecolor{cIndigo}{HTML}{8a58da}
\definecolor{cDeepBlue}{HTML}{4a8ce8}
\definecolor{cLightBlue}{HTML}{43e0ef}
\definecolor{cGreen}{HTML}{3ef450}

\tikzset{every node/.style={
    fill=black,
    draw=black,
    text=white,
    inner sep=1pt,
    minimum width=5mm,
    font=\small,
    line width=1.5pt,
    circle
}}
\tikzset{not_vert/.style={
    fill=white,
    text=black,
    draw opacity=0,
    minimum width=0
}}
\tikzset{every path/.style={
    line width=2pt
}}

\begin{document}
\title{Aharoni's rainbow cycle conjecture holds up to an additive constant}
\author{Patrick Hompe}
\author{Tony Huynh}
\address[Patrick Hompe]{}
\email{patrickhompe@gmail.com}
\address[Tony Huynh]{Dipartimento di Informatica, Sapienza Università di Roma, Italy}
\email{huynh@di.uniroma1.it}

\address[]{}
\email{}
\thanks{}
\begin{abstract}
In 2017, Aharoni proposed the following generalization of the Caccetta-H\"{a}ggkvist conjecture: if $G$ is a simple $n$-vertex edge-colored  graph with $n$ color classes of size at least $r$, then $G$ contains a rainbow cycle of length at most $\lceil n/r \rceil$. 

In this paper, we prove that, for fixed $r$, Aharoni's conjecture holds up to an additive constant. Specifically, we show that 
for each fixed $r \geq 1$, there exists a constant $\alpha_r \in O(r^5 \log^2 r)$ such that if $G$ is a simple $n$-vertex edge-colored graph with $n$ color classes of size at least $r$, then $G$ contains a rainbow cycle of length at most $n/r + \alpha_r$.
\end{abstract}
\maketitle

\section{Introduction}
A digraph $D$ is \defn{simple} if for all distinct vertices $u,v \in V(D)$, there is at most one arc from $u$ to $v$.
Perhaps the most famous open problem concerning digraphs is the following conjecture of 
Caccetta and H\"aggkvist from 1978.

\begin{conjecture}[\cite{CH78}] \label{conj-ch}
Let $D$ be a simple $n$-vertex digraph with minimum out-degree at least $r$. Then $D$ contains a directed cycle of length at most $\lceil n/r \rceil$.
\end{conjecture}

Although the \CH{}conjecture has attracted considerable attention over the years, proving it still remains out of reach. We highlight here one partial result of~\citet{Shen02}, which is in the spirit of our paper.

\begin{theorem}\label{thm-shen1}
Let $D$ be a simple $n$-vertex digraph with minimum out-degree at least $r$. Then $D$ contains a directed cycle of length at most $\lceil n/r \rceil + 73$.
\end{theorem}

Let $G$ be an edge-colored graph.  We say that $G$ is \defn{simple} if no color class contains parallel edges.  A subgraph of $G$ is \defn{rainbow} if it does not contain two edges of the same color.  
In this paper, we focus on the following generalization of the \CH{}conjecture, due to Aharoni~\cite{ADH19}.

\begin{conjecture}\label{conj-aharoni}
Let $G$ be a simple $n$-vertex edge-colored graph with $n$ color classes of size at least $r$. Then $G$ contains a rainbow cycle of length at most $\lceil n/r \rceil$.
\end{conjecture}

It is well-known that Aharoni's conjecture implies the \CH{}conjecture (see for example~\cite{DDFGGHMM21} or~\cite{CGHI22}).  For completeness, we include the proof here. 
\begin{proof}[Proof of~\cref{conj-ch} assuming~\cref{conj-aharoni}]
    Let $D$ be a simple $n$-vertex digraph with minimum out-degree at least $r$. We create an edge-coloured graph $G$ from $D$ with $V(G)=V(D)$ as follows.  For each arc $(u,v) \in A(D)$ we create an edge $uv \in E(G)$ with color $u$.  Note that $G$ has $n$ color classes of size at least $r$ (since $D$ has minimum out-degree at least $r$).  Moreover, since $D$ is a simple digraph it follows that $G$ is a simple edge-colored graph. By~\cref{conj-aharoni}, $G$ contains a rainbow cycle of length at most $\lceil n/r \rceil$.  This rainbow cycle necessarily corresponds to a directed cycle in $D$.  
\end{proof}

Note that~\cref{conj-aharoni} holds when $r=1$, since we can choose one edge from each color class and use the easy fact that every $n$-vertex graph with $n$ edges contains a cycle.  The $r=2$ case of~\cref{conj-aharoni} was proven by~\citet{DDFGGHMM21}.

\begin{theorem}[\cite{DDFGGHMM21}]
Let $G$ be a simple $n$-vertex edge-colored graph with $n$ color classes of size at least 2. Then $G$ contains a rainbow cycle of length at most $\lceil n/2 \rceil$.
\end{theorem}

The $r=3$ case is still open, but the following approximate result was obtained by~\citet{CGHI22}.

\begin{theorem}[\cite{CGHI22}]\label{r=3-paper}
Let $G$ be a simple $n$-vertex edge-colored graph with $n$ color classes of size at least $3$. Then $G$ contains a rainbow cycle of length at most $4n/9+7$.
\end{theorem}

For general $r$, the best bound is due to~\citet{HS22}, who proved that Aharoni's conjecture holds up to a multiplicative constant.

\begin{theorem}[\cite{HS22}] \label{thm:constantfactor}
Let $c = 10^{11}$ and $G$ be a simple $n$-vertex edge-colored graph with $n$ color classes of size at least $cr$. Then $G$ contains a rainbow cycle of length at most $\lceil n/r \rceil$.
\end{theorem}

In this paper, we strengthen~\cref{r=3-paper} and~\cref{thm:constantfactor}, by proving that Aharoni's conjecture holds for each fixed $r$ up to an additive constant. In fact, we prove a stronger `defect' version  which allows some color classes to have size less than $r$.  To state our main result precisely, for an edge-colored graph $G$, let $C(G)$ be the set of colors appearing in $E(G)$, and for a color $c \in C(G)$, let $c_G$ be the edges of $E(G)$ with color $c$. For each $r \in \mathbb{N}$, we define
\[
\df_r(G)\coloneqq \sum_{c \in C(G), |c_G| \leq r} (r-|c_G|).
\]

The following is our main theorem.  

\begin{theorem}\label{thm:main-gen}
For all $r \geq 2$, there exists a constant $\alpha_r \in O(r^5 \log^2 r)$ such that if $G$ is a simple $n$-vertex edge-colored graph with $n$ color classes of size at least $2$ and at most $r$, then $G$ contains a rainbow cycle of length at most
$$\frac{n+\df_r(G)}{r} + \alpha_r.$$
\end{theorem}

As an immediate corollary of~\cref{thm:main-gen}, we obtain the result claimed in the abstract.  

\begin{corollary} \label{thm:main}
For all $r \geq 1$, there exists a constant $\alpha_r \in O(r^5 \log^2 r)$ such that if $G$ is a simple $n$-vertex edge-colored graph with $n$ color classes of size at least $r$, then $G$ contains a rainbow cycle of length at most
$$\frac{n}{r} + \alpha_r.$$
\end{corollary}

\begin{proof}
 Since we have already noted that~\cref{conj-aharoni}  holds when $r=1$, we may assume $r \geq 2$. Next, let $G$' be obtained from $G$ by deleting edges so that each color class has size exactly $r$.  By applying~\cref{thm:main-gen} to $G'$ and noting that $\df_r(G')=0$, it follows that $G'$ (and hence also $G$) contains a rainbow cycle of length at most $\frac{n}{r} + \alpha_r$.
\end{proof}

Our proof shows that we may take $\alpha_r =3000000 r^5 \log^2(r)$.  An interesting open problem is to improve the dependence on $r$.  Given that the \CH{}conjecture is still open, likely the best one could hope for is $\alpha_r \leq \ell$ for some universal constant $\ell$.  This would be a generalization of~\cref{thm-shen1}, up to an additive constant.  

We would also like to point out that there have been many recent results on Aharoni's conjecture. These are too numerous to list here, but  we refer the interested reader to~\cite{CGHI22} for a survey of the state of the art.

\subsection*{Paper Outline} 
In~\cref{sec:tools} we state some preliminary lemmas needed in the proof of our main theorem.  
Our proof is split into two cases, depending on whether there are many `non-star' vertices or few non-star vertices.  We set-up our notation and assumptions in~\cref{sec:setup}, where we also prove a key lemma used in the proof. Then, we handle the many non-star case in \cref{sec:many}, and the few non-star case in~\cref{sec:few}.

\section{Preliminaries} \label{sec:tools}
We begin with some results from the literature needed in the proof of our main theorem.
We say that a graph is \defn{excess-$k$} if it contains at least $k$ more edges than vertices.  \citet{BS02} proved the following upper bound on the girth of excess-$k$ graphs.

 \begin{theorem} \label{thm:BS}
For all $n \geq 4$ and $k \geq 2$, every $n$-vertex, excess-$k$ graph has girth at most
\[
\frac{2(n+k)}{3k} (\log k + \log \log k +4).
\]
 \end{theorem}

In~\cref{thm:BS}, and for this paper, all logs are in base 2. We will use the following immediate corollary of~\cref{thm:BS} (see~\cite{HS22}).

\begin{corollary} \label{cor:girth}
For all $n \geq 4$ and $k \geq 2$, every $n$-vertex, excess-$k$ graph has girth at most
\[
\frac{14(n+k)\log k}{3k}.
\]
\end{corollary}

  Now, let $D$ be a digraph.  A \defn{sink} is a vertex with out-degree equal to zero.  For each $r \in \mathbb{N}$, we define
\[
\df_r(D)\coloneqq \sum_{u \in U} (r-\deg^+(u)),
\]
where $U$ is the set of vertices of $D$ of out-degree at most $r$.  We require the following theorem of~\citet{shen2000}.

\begin{theorem}\label{thm:shen}
Let $D$ be a simple $n$-vertex digraph with no sink, and let $g$ be the length of a shortest directed cycle of $D$. If $g \ge 2r-1$, then $n \ge r(g-1) + 1 - \df_r(D)$.
\end{theorem}

We will use the following immediate corollary of \cref{thm:shen}.
\begin{corollary}\label{cor:shen}
Let $D$ be a simple $n$-vertex digraph with no sink, and let $g$ be the length of a shortest directed cycle of $D$. Then
$$g \leq \frac{n+\df_r(D)}{r} + 2r-1.$$
\end{corollary}
\begin{proof}
The statement is obviously true if $g \leq 2r-2$.  Thus, we may assume that $g \geq 2r-1$.    
It follows by \cref{thm:shen} that
$$
g \leq \frac{n+\df_r(D)-1}{r} + 1 \leq \frac{n+\df_r(D)}{r}+2r-1,
$$
as desired. 
\end{proof}

\section{The Set-up} \label{sec:setup}
For the remainder of the paper, $G$ is a simple edge-colored graph with $n$ vertices, $n$ colors, and each color class of size at least $2$ and at most $r$, for $r \ge 2$. In other words, $G$ satisfies the hypotheses of~\cref{thm:main-gen}. Let $k := \lceil 78 r \log{r} \rceil$ and $f(r) = 80k^2 r^2$.  We will show that, for fixed $r$, we may take $\alpha_r := (4r+2)f(r)+2r$ in~\cref{thm:main-gen}.  

To prove~\cref{thm:main-gen}, we proceed by induction on $n$. The base case is $n \le \alpha_r$, which is true since there is always a rainbow cycle of length at most $n$. Thus, for the remainder of the paper, we also assume that $n>\alpha_r$.


For a subset of vertices $X \subseteq V(G)$, we let $G[X]$ be the induced subgraph of $G$ on $X$, and $C(X)$ be the set of colors which appear in $G[X]$. By our earlier definition of $C(G)$ for an edge-colored graph $G$, an alternative way of writing this definition is $C(X) = C(G[X])$.

Let $U \subseteq V(G)$.  A \defn{galaxy rooted at $U$} is a collection $M:=\{M_u : u \in U\}$ of subgraphs of $G$ such that
\begin{itemize}
\item
$M_u$ is a star centred at $u$, for all $u \in U$,
\item
$|E(M_u)| \geq 1$ and all edges of $M_u$ are the same color $c_u$,
\item
$c_u \neq c_v$, for all $u \neq v$.
\end{itemize}
We say that $x \in V(G)$ is an \defn{$M$-neighbor} of $u \in U$ if $x$ is a neighbor of $u$ in $M_u$. We let $C(M):=\{c_u : u \in U\}$ be the set of \defn{$M$-colors}.  Finally, for each $u \in U$ we let $\abs{u}_M:=|E(M_u)|$.  We will use the following key lemma multiple times in the proof.

\begin{lemma}\label{key-lemma}
Let $R = X \cup \{u_1, \cdots, u_m\} \subseteq V(G)$, $M:=\{M_{u_i} : i \in [m]\}$ be a galaxy rooted at $\{u_1, \cdots, u_m\}$, $C_X \subseteq C(X)$, and  $d \in \R$ be such that 
\begin{enumerate}
\item
$V(M_{u_i}) \setminus \{u_i\} \subseteq X \cup \{u_1, \dots, u_{i-1}\}$ for all $i \in [m]$,
\item
$C_X \cap C(M)=\emptyset$,
\item
for every pair of vertices $u,v \in X$, there exists a rainbow path consisting of colors in $C_X$  from $u$ to $v$ in $G[X]$ of length at most $d$.
\end{enumerate}
Then for every pair of vertices $u,v \in R$, there exists a rainbow path consisting of colors in $C_X \cup C(M)$ from $u$ to $v$ in $G[R]$  of length at most
\begin{align*}
\frac{m+\sum_{i=1}^m (r-\abs{u_i}_M)}{r} + d + 2.
\end{align*}
\end{lemma}
\begin{proof} See~\cref{helpfulfigure} for an example with $m=7$.
For each $u \in R$, we define a function $p: R \to R$ as follows.  If $u \in X$, then $p(u)=u$. If $u \notin X$, but some $M$-neighbor $v$ of $u$ is in $X$, we arbitrarily choose one such $v$, and define $p(u)=v$. Otherwise, define $p(u)=u_i$, where $i$ is the minimum index such that $u_i$ is an $M$-neighbor of $u$.  For example, in~\Cref{helpfulfigure}, we have $p(u_4)=u_2$.  

Let $u,v \in R$.  Let $P^u$ be the path beginning at $u$ and repeatedly applying the map $p$ until we reach a vertex in $X$.  Define $P^v$ analogously, but starting from $v$. 

First suppose there exists $u_a \in V(P^u)$ and $u_b \in V(P^v)$ such that $u_a$ and $u_b$ have a common $M$-neighbor in $\{u_1, \dots, u_m\}$.  Choose $u_a$ and $u_b$ such that $a+b$ is maximum. By symmetry, we may assume $a > b$.      
Let $P_a$ be the subpath of $P^u$ from $u$ to $u_a$, and let $P_b$ be the subpath of $P^v$ from $v$ to $u_b$. Let $Q_a = P_a \setminus \{u_a\}$ and $Q_b = P_b \setminus \{u_b\}$. By the maximality of $a+b$, 
\begin{align*}
m &\geq 1 + \sum_{w \in V(Q_a)} |w|_M + \sum_{w \in V(Q_b)} |w|_M \\
&= 1+\sum_{w \in V(Q_a)} (r-(r-|w|_M)) + \sum_{w \in V(Q_b)} (r-(r-|w|_M)).
\end{align*}
Rearranging, we have 
\begin{align*}
|E(P_b)|+|E(P_a)| &\leq \frac{m -1 + \sum_{w \in V(Q_a)} (r-|w|_M) + \sum_{w \in V(Q_b)} (r-|w|_M)}{r} \\
&\leq \frac{m+\sum_{i=1}^m (r-|u_i|_M)}{r}.
\end{align*}
Let $u_\ell$ be a common $M$-neighbor of $u_a$ and $u_b$ in $\{u_1, \dots, u_m\}$.  Then $P_a \cup P_b \cup \{u_au_\ell, u_bu_\ell \}$ is  a rainbow path consisting of colors in $C_X \cup C(M)$ from $u$ to $v$ in $G[R]$  of length at most
$$\frac{m+\sum_{i=1}^m (r-|w_i|_M)}{r} + d+2,$$
as required.  

For the remaining case, let $u' \in X$ and $v' \in X$ be the other ends of $P^u$ and $P^v$.  Let $Q^u=P^u \setminus u'$ and $Q^v=P^v \setminus v'$, and let $S^u$ be $Q^u$ minus its last vertex and $S^v$ be $Q^v$ minus its last vertex. If both $u$ and $v$ are in $X$, then we are done by (3).  Thus, by symmetry, we may assume that $u = u_a$ for some $a \in [m]$ and that either $v \in X$ or $v =u_b$ for some $b<a$. By the previous case, we may assume that for all $x \in V(Q^u)$ and $y \in V(Q^v)$, $x$ and $y$ do not have any common $M$-neighbors in $\{u_1, \dots, u_m\}$.  Therefore, 
\begin{align*}
m &\geq 1 + \sum_{w \in V(S_a)} |w|_M + \sum_{w \in V(S_b)} |w|_M \\
&= 1+\sum_{w \in V(S_a)} (r-(r-|w|_M)) + \sum_{w \in V(S_b)} (r-(r-|w|_M)).
\end{align*}
Rearranging, we have 
\begin{align*}
|E(P_u)|+|E(P_v)| &\leq |E(Q_u)|+|E(Q_v)| +2\\
&\leq \frac{m -1 + \sum_{w \in V(S_u)} (r-|w|_M) + \sum_{w \in V(S_v)} (r-|w|_M)}{r}+2 \\
&\leq \frac{m+\sum_{i=1}^m (r-|u_i|_M)}{r}+2.
\end{align*}
By (3), there is a rainbow path $P$ consisting of colors in $C_X$ from $u'$ to $v'$ in $G[X]$ of length at most $d$. Since $C_X \cap C(M)=\emptyset$, $P^u \cup P^v \cup P$ is a rainbow path consisting of colors in $C_X \cup C(M)$ from $u$ to $v$ in $G[R]$  of length at most
$$\frac{m+\sum_{i=1}^m (r-|u_i|_M)}{r}+d+2,$$
as required. This completes the proof of the lemma.
\end{proof}

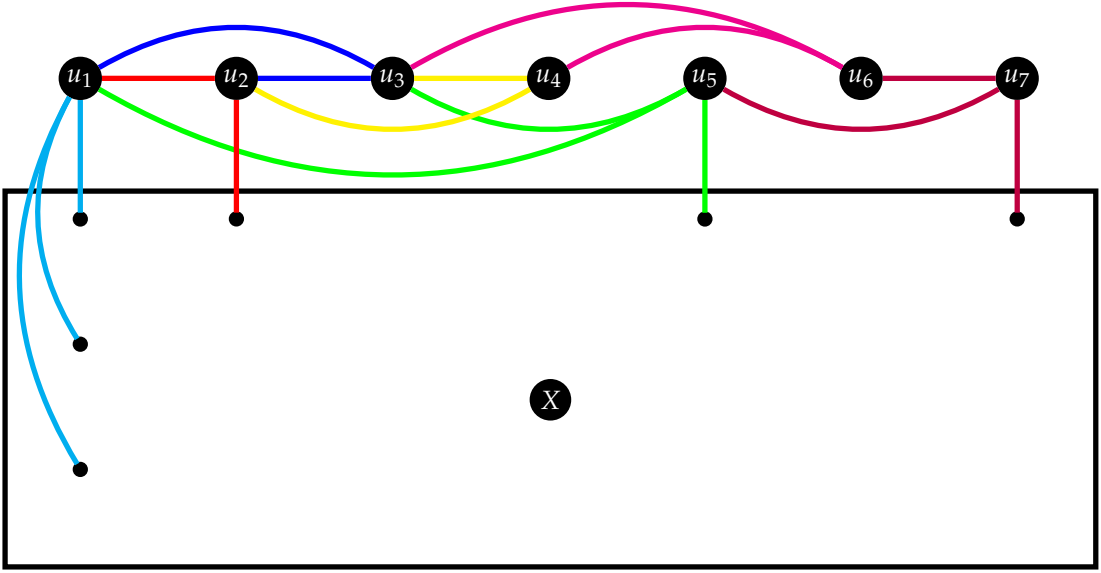
\begin{figure} 
    \centering
 \begin{tikzpicture}[node distance=1.5cm]
 \draw (-1,-6.5) rectangle node[below] {$X$} +(14.5,5);

  \node[circle,scale=1]          (A)              {$u_1$};
  \node[circle,scale=1]            (B) [right=of A] {$u_2$};
  \node[circle,scale=1]            (C) [right=of B] {$u_3$};
  \node[circle,scale=1]             (D) [right=of C] {$u_4$};
  \node[circle,scale=1]           (E) [right=of D] {$u_5$};
  \node[circle,scale=1]             (F) [right=of E] {$u_6$};
 \node[circle,scale=1]             (G) [right=of F] {$u_7$};
 \node[fill, circle, scale=0.3]       (H) [below=of G] {};
  \node[fill, circle, scale=0.3]       (I) [below=of E] {};
  \node[fill, circle, scale=0.3]       (J) [below=of B] {};
  \node[fill, circle, scale=0.3]       (K) [below=of A] {};
\node[fill, circle, scale=0.3]       (L) [below=of K] {};
\node[fill, circle, scale=0.3]       (M) [below=of L] {};

  \path (A) edge [red, above]     (B);
  \path (B) edge [blue, above]     (C);
  \path (C) edge [yellow, above]     (D);
   \path (F) edge [purple, above]     (G);
   \path (E) edge [purple, bend right]  (G);
    \path (G) edge [purple]  (H);
    \path (C) edge [magenta, bend left]  (F);
    \path (D) edge [magenta, bend left]  (F);
    \path (E) edge [green]  (I);
    \path (A) edge [green, bend right]  (E);
    \path (C) edge [green, bend right]  (E);
    \path (B) edge [yellow, bend right]  (D);
    \path (A) edge [blue, bend left]  (C);
    \path (B) edge [red]  (J);
    \path (A) edge [cyan]  (K);
    \path (A) edge [cyan, bend right]  (L);
    \path (A) edge [cyan, bend right]  (M);
    \end{tikzpicture}
    \caption{An example of an edge-colored graph satisfying the hypotheses of Lemma \ref{key-lemma}. Unlabelled vertices are the neighbours of $\{u_1, \dots, u_7\}$ in $X$.  Other vertices of $X$ and edges with both ends in $X$ are not depicted.}
   \label{helpfulfigure}
    \end{figure}

\section{Many Non-Stars}
\label{sec:many}
A color class of $G$ is a \defn{star class} if its edges form a star. A vertex of $G$  is  a \defn{star vertex} if it is the centre of a star class, and is otherwise a \defn{non-star vertex}. 
Let $N$ be the set of non-star vertices of $G$.  

Recall that $k = \lceil 78 r \log{r} \rceil$, $f(r) = 80k^2 r^2$, and $\alpha_r = (4r+2)f(r)+2r$.
In this section we consider the case where $|N| > f(r)$.  We say that a subset $N' \subseteq N$ \defn{covers} a color class $A$ if every edge in $A$ is incident to at least one vertex in $N'$.   

\begin{claim} \label{claim:dominate}
Each color class is covered by at most $4\binom{|N|}{k-2}$ sets $N' \subseteq N$ of size $k$.
\end{claim}
\begin{proof}
First suppose $A$ is a star class.  Then $A$ is covered by at most $$\binom{|N|-|A|}{k-|A|} \leq \binom{|N|-2}{k-2} \leq 4\binom{|N|}{k-2}$$
sets $N' \subseteq N$ of size $k$ (since $|A| \geq 2$).

So, suppose instead that $A$ is a non-star class. If there exists a matching $M = \{u_1v_1, u_2 v_2\}$ of size $2$ in $A$, then we note that every set of vertices which covers $A$ must contain either $u_1$ or $v_1$, and either $u_2$ or $v_2$. It follows that $A$ is covered by at most
$$4\binom{|N|}{k-2}$$ subsets of size $k$ in $N$, as required. 

Finally, suppose that there does not exist a matching of size 2 in $A$. Since $A$ is not a star, it follows that $A$ is a triangle. Then the number of subsets of $N$ of size $k$ that cover $A$ is at most

\begin{align*}
    3 \binom{|N|-2}{k-2} \le 4 \binom{|N|}{k-2}.
\end{align*}

This completes the proof.
\end{proof}

\begin{lemma}\label{lemma-manystars}
If $|N| > f(r)$, then $G$ contains a rainbow cycle of length at most
$$\frac{n+\sum_{c \in C(G)}(r-|c_G|)}{r}+\alpha_r.$$
\end{lemma}

\begin{proof}
First suppose that fewer than $\frac{(|N|-k)^2}{4k^2}$ color classes contain an edge with an end in $N$. By~\cref{claim:dominate}, at most
\begin{align*}
\frac{(|N|-k)^2}{4k^2} \cdot 4\binom{|N|}{k-2} < \binom{|N|}{k}
\end{align*}
subsets of $N$ of size $k$ cover some color class. Therefore, there exists a subset $N' \subseteq N$ of size $k$ such that $N'$ does not cover any color class.  By choosing an edge of each color with neither end in $N'$ we obtain a rainbow subgraph $H$ of $G$ with $|E(H)|=n$ and $|V(H)| \leq |V(G) \setminus N'|=n-k$.  In particular, $H$ is excess-$k$. Since $r \ge 2$ and $n \ge k$ (otherwise we have a rainbow cycle of length at most $k < \alpha_r$), ~\cref{cor:girth} gives that $H$ has a rainbow cycle of length at most 
\begin{align*}
\frac{14(n+k)\log k}{3k} &\leq \frac{28n\log{k}}{3k} \\
 &\leq \frac{28n(\log{78}+\log{r}+\log{\log{r}})}{234r \log{r}} \\
&\leq \frac{28n(\log{78}+2)\log{r}}{234r\log{r}} \\
&\leq \frac{n}{r},
\end{align*}
as desired.

So, we may assume that at least $\frac{(|N|-k)^2}{4k^2}$ color classes have at least one edge with an end in $N$. Note that $|N|-k \ge |N|/2$ since $|N| > f(r) \ge 2k$. Therefore, by averaging, there exists a vertex $v \in N$ with at least $\frac{|N|}{16k^2} > \frac{f(r)}{16k^2}=5r^2$ colors having at least one edge incident with $v$. Therefore, there exists a rainbow star $T$ centred at $v$ with $|V(T)| > 5r^2. $

Now, let $X$ be a maximal set of vertices containing $V(T)$ such that 
\begin{enumerate}
\item
exactly $|X|-1$ colors of $G$ appear in $G[X]$,
\item
for all $u,v \in X$, there exists a rainbow path from $u$ to $v$ in $G[X]$ with length at most
$$\frac{|X|+\sum_{c \in C(X)} (r-|c_G|)}{r} -5r +2,$$
\end{enumerate}
where $|c_G|$ is the size of the color class in $G$, and $C(X)$ is the set of colors which appear in $G[X]$.

Note that $V(T)$ satisfies the first condition since otherwise $G[V(T)]$ contains a rainbow triangle and we have obtained a short rainbow cycle. Moreover, $V(T)$ satisfies the second condition since $|V(T)| > 5r^2$ and  every color class has size at most $r$. Since $V(T)$ satisfies both conditions and clearly $V(T) \subseteq V(T)$, it follows that $V(T)$ is a valid candidate and therefore $X$ is a well-defined set.

Now, let $u_1, \cdots, u_m$ be a maximal sequence of vertices in $G \setminus X$ such that there is a star color class $A_i$ centred at $u_i$ with $V(A_i) \setminus \{u_i\} \subseteq X \cup \{u_1, \dots, u_{i-1}\}$. Note that $u_1, \dots, u_m$ exist, since we allow $m=0$. Let $T' = X \cup \{u_1, \cdots, u_m\}$, $c_i$ denote the color of $A_i$, and $C'(T') = C(X) \cup \{c_1, \cdots, c_m\}$. Let $d:=\frac{|X|+\sum_{c \in C(X)} (r-|c_G|)}{r} -5r +2$. By \cref{key-lemma},  for all $u,v \in T'$ there exists a rainbow path in $G[T']$ from $u$ to $v$ consisting of colors in $C(X) \cup \{c_1, \dots, c_m\}$ of length at most

\begin{align*}
\frac{m+\sum_{i=1}^m (r-\abs{u_i}_M)}{r} + d + 2 
&= \frac{m+\sum_{i=1}^m (r-\abs{u_i}_M)}{r} + \left(\frac{|X|+\sum_{c \in C(X)} (r-|c_G|)}{r} -5r +2\right) + 2 \\
&= \frac{|T'|+\sum_{c \in C'(T')} (r-|c_G|)}{r} -5r +4.
\end{align*}

It follows that exactly $|X|-1+m=|T'|-1$ colors appear in $G[T']$; otherwise, $G[T']$ contains a rainbow cycle of length at most 
$$\frac{|T'|+\sum_{c \in C'(T')} (r-|c_G|)}{r} -5r +5 \leq \frac{n + \df_r(G)}{r}-5r+5 \leq \frac{n + \df_r(G)}{r}.$$

Therefore, $C'(T') = C(T')$. Now, let $\mathcal{A}$ be the set of color classes $A$ such that no edge of $A$ is in $G[T']$ and some edge in $A$ has exactly one end in $T'$.  
For each $A \in \mathcal{A}$, choose an edge $e_A \in A$ with exactly one end in $T'$, and let $Y \subseteq V(G) \setminus T'$ be the set of ends of these edges not in $T'$. Note that $|Y|=|\mathcal{A}|$; otherwise $G[T' \cup Y]$ contains a rainbow cycle of length at most
$$\frac{|T'|+\sum_{c \in C(T')} (r-|c_G|)}{r} -5r +6 \leq \frac{n+\df_r(G)}{r}-5r+6 \leq \frac{n+\df_r(G)}{r}.$$

Moreover, exactly $|T' \cup Y|-1$ colors appear in $G[T' \cup Y]$; otherwise $G[T' \cup Y]$ contains a rainbow cycle of length at most 
$$\frac{|T'|+\sum_{c \in C(T')} (r-|c_G|)}{r} -5r +7 \leq \frac{n+\df_r(G)}{r}-5r+7 \leq \frac{n+\df_r(G)}{r}.$$

Note that when adding $Y$ to $T'$ the `rainbow diameter' increases by at most $2$, regardless of the size of $Y$.  Therefore, for all $u,v \in T' \cup Y$, there exists a rainbow path from $u$ to $v$ in $G[T' \cup Y]$ of length at most 

$$\left(\frac{|T'|+\sum_{c \in C(T')} (r-|c_G|)}{r} -5r +4\right)+2 \leq \frac{|T' \cup Y|+\sum_{c \in C(T' \cup Y)} (r-|c_G|)}{r} -5r +6-\frac{|Y|}{r}.$$

It follows that $|Y| \leq 4r-1$; otherwise $T' \cup Y$ contradicts the maximality of $X$. 
Now, we contract $T'$ to a single vertex, delete all edges with colors in $T'$, and remove parallel edges of the same color. Let the resulting edge-colored graph be $G'$. Note that by the construction of $T'$, $G'$ contains $|V(G')|$ color classes of size at least $2$, and at most $4r-1$ color classes have less edges in $G'$ than they do in $G$ (since $|\mathcal{A}| \leq 4r-1$). By induction on $n$, $G'$ has a rainbow cycle $D$ of length at most
$$\frac{|V(G')|+\sum_{c \in C(G')} (r-|c_{G'}|)}{r}+\alpha_r \leq \frac{|V(G')|+(4r-1)(r-2)+\sum_{c \in C(G)}(r-|c_G|)}{r}+\alpha_r.$$

 Observe that $D$ is either a rainbow cycle in $G$, or a rainbow path between $u$ and $v$, for some $u,v \in T'$. If the former holds, then 
\[
|V(D)| \leq \frac{n+\sum_{c \in C(G)}(r-|c_G|)}{r}+\alpha_r,
\]
since $|V(G')|+(4r-1)(r-2) \leq |V(G')|+5r^2-1 \leq |V(G')|+|T'|-1 = n$.  

In the latter case, from our earlier application of \cref{key-lemma}, we have that there exists a rainbow path $P$ in $G[T']$ from $u$ to $v$ of length at most

$$\frac{|T'|+\sum_{c \in C(T')} (r-|c_G|)}{r} -5r +4.$$

By the construction of $G'$, the colors used in $P$ are disjoint from the colors used in $D$.  Therefore, $D \cup P$ is a rainbow cycle in $G$ of length at most 
\begin{align*}
|E(D)|+|E(P)| &\leq \frac{|V(G')|+|T'|+(4r-1)(r-2)+\sum_{c \in C(G)}(r-|c_G|)}{r}-5r+4+\alpha_r \\
&=\frac{n+1+(4r-1)(r-2)+\sum_{c \in C(G)}(r-|c_G|)}{r}-5r+4+\alpha_r \\
&\leq \frac{n+\sum_{c \in C(G)}(r-|c_G|)}{r}+\alpha_r, 
\end{align*}
since $1+(4r-1)(r-2) \leq r(5r-4)$ simplifies to $r^2+5r \geq 3$ which is true for $r \ge 2$. This completes the proof of the lemma.
\end{proof}

\section{Few Non-Stars} \label{sec:few}
Recall from the previous section that $N$ is the set of non-star vertices of $G$.  
In this section, we complete the proof by handling the case $|N| \le f(r)$.

\begin{lemma}\label{lemma-fewstars}
If $|N| \le f(r)$, then $G$ contains a rainbow cycle of length at most
$$\frac{n+\sum_{c \in C(G)}(r-|c_G|)}{r}+\alpha_r.$$
\end{lemma}

\begin{proof}
Let $N = \{v_1, v_2, \cdots, v_t\}$ be an arbitrary fixed ordering of the vertices of $N$.  Let $U=V(G) \setminus N$ and choose a galaxy $M:=(M_u : u \in U)$ rooted at $U$ by arbitrarily choosing a star class $M_u$ of color $c_u$ centred at $u$ for each $u \in U$. An edge $e \in E(G)$ is an \defn{$M$-edge} if $e \in E(M_u)$ for some $u \in U$.   For $Y \subseteq V(G)$ we define $\gamma(Y)$ to be the set of star vertices $u \in V(G) \setminus Y$ such that at least one $M$-neighbor of $u$ is in $Y$.  

For each $X \subseteq V(G)$, let $C_M(X)$ be the set of $M$-colors which have an edge with both ends in $X$. Now, we iteratively construct pairwise-disjoint sets of vertices $T_1, T_2, \cdots, T_t$ of $G$ and associated constants $d_1,d_2, \cdots, d_t$ such that for all $j \in [t]$, we have the following four properties:\\

\begin{enumerate} 
\item \label{item1}
$T_j \cap N = \{v_j\}$.\\
\item \label{item2}
For each star vertex $v \notin \bigcup_{i=1}^j T_j $, $v$ has at least one $M$-neighbor which is not in $\bigcup_{i=1}^j T_j$.\\
\item \label{item3}
For each pair of vertices $u,v \in T_j$, there exists a rainbow path of $M$-edges between $u$ and $v$ in $G[T_j]$ of length at most $$\frac{|T_j|+\sum_{c \in C_M(T_j)} (r-|c_G|)}{r}+d_j.$$\\
\item \label{item4}
$\sum_{i=1}^j d_i  + |\gamma(\bigcup_{i=1}^j T_i)| \leq (4r+1)j$.\\
\end{enumerate}

Fix $j \in [t]$, and suppose we have already constructed $T_1, \cdots, T_{j-1}$ with the associated constants $d_1,\cdots,d_{j-1}$ such that they satisfy the above four properties. Let $G_j:=G \setminus \bigcup_{i=1}^{j-1} T_i$.  For $X \subseteq V(G_j)$, let $s(X)$ be the set of star vertices $x \in X$ such that at least one $M$-neighbor of $x$ is in $\bigcup_{i=1}^{j-1} T_i$.  
Let $X$ be a maximal set of vertices of $G_j$ such that $X \cap N=\{v_j\}$, and for each pair of vertices $u,v \in X$, there exists a rainbow path of $M$-edges between $u$ and $v$ in $G[X]$ of length at most $$d:=\frac{|X|+\sum_{c \in C_M(X)}(r-|c_G|)}{r}+|s(X)|.$$

Note that $X$ exists since $\{v_j\}$ is a candidate for $X$. Let $u_1, \dots, u_m$ be a maximal sequence of vertices in $G_j \setminus (X \cup N)$ such that for all $i \in [m]$,  all the $M$-neighbors of $u_i$ are in $\bigcup_{i=1}^{j-1} T_i \cup X \cup \{u_1, \dots, u_{i-1}\}$. Define $T_j:=X \cup \{u_1, \dots, u_m\}$, and $d_j:=|s(T_j)|+2$.

Clearly, $T_j$ satisfies~\cref{item1} and~\cref{item2} (by the maximality of $u_1, \dots, u_m$). Let $C':=\{c_{u_i} : i \in [m]\}$. By \cref{key-lemma},  for any pair of vertices $u,v \in T_j$, there is a rainbow path of $M$-edges from $u$ to $v$ in $G[T_j]$ of length at most
\begin{align*}
\frac{m+\sum_{c \in C'} (r-|c_G|)}{r} + d + 2 
&=
\frac{m+\sum_{c \in C'} (r-|c_G|)}{r} + \left(\frac{|X|+\sum_{c \in C_M(X)}(r-|c_G|)}{r}+|s(X)|\right)+2 \\
&\le \frac{|T_j|+\sum_{c \in C_M(T_j)} (r-|c_G|)}{r}+|s(T_j)|+2 \\
&= \frac{|T_j|+\sum_{c \in C_M(T_j)} (r-|c_G|)}{r}+d_j,
\end{align*}
since $|s(X)| \le |s(T_j)|$. It follows that \cref{item3} is satisfied. Finally, we show that~\cref{item4} holds.

\begin{claim}
$\sum_{i=1}^j d_i  + |\gamma(\bigcup_{i=1}^j T_i)| \leq (4r+1)j$.
\end{claim}

\begin{proof}
Let 
\[
Y:=\{u \notin \bigcup_{i=1}^j T_i : \text{at least one $M$-neighbor of $u$ is in $T_j$}\}.
\]

Note that if we add $Y$ to $T_j$, then the `rainbow diameter' increases by at most $2$, regardless of the size of $Y$.  Thus, $|Y| \leq 4r-1$; otherwise, $Y \cup T_j$ contradicts the maximality of $X$ (see Lemma \ref{lemma-manystars} for an analogous argument with more detail).  Moreover, by definition, $s(T_j) \subseteq \gamma(\bigcup_{i=1}^{j-1} T_i)$ and $\gamma(\bigcup_{i=1}^j T_i) \subseteq Y \cup ( \gamma(\bigcup_{i=1}^{j-1} T_i) \setminus s(T_j))$. Thus, $|\gamma(\bigcup_{i=1}^j T_i)| \leq |Y|+|\gamma(\bigcup_{i=1}^{j-1} T_i)|-|s(T_j)|$.
It follows that, 
\begin{align*}
\sum_{i=1}^j d_i  + |\gamma(\bigcup_{i=1}^j T_i)| &\leq d_j+ |Y|-|s(T_j)| + \sum_{i=1}^{j-1} d_i +|\gamma(\bigcup_{i=1}^{j-1} T_i)| \\
&= 2+|Y| + \sum_{i=1}^{j-1} d_i + |\gamma(\bigcup_{i=1}^{j-1} T_i)| \\
&\leq (4r+1) + \sum_{i=1}^{j-1} d_i + |\gamma(\bigcup_{i=1}^{j-1} T_i)| \\
&\leq (4r+1) +(4r+1)(j-1) \\
&= (4r+1)j. \qedhere
\end{align*}
\end{proof}

To finish, we have two cases. First, suppose that $V(G) = \bigcup_{i=1}^t T_i$. Let $A$ be the set of edges of $G$ which are not $M$-edges. Note that exactly $|N|$ colors appear in $A$. If some $e \in A$ has both ends in some $T_i$, then we obtain a rainbow cycle of length at most
\begin{align*}
1+\frac{|T_i|+\sum_{c \in C_M(T_i)}(r-|c_G|)}{r}+d_i &\leq 1+ \frac{n+\sum_{c \in C(G)}(r-|c_G|)}{r} + (4r+1)f(r) \\
&\leq \frac{n+\sum_{c \in C(G)}(r-|c_G|)}{r} + (4r+2)f(r) \\
&\leq \frac{n+\sum_{c \in C(G)}(r-|c_G|)}{r} + \alpha_r.
\end{align*}

So we may assume that every edge in $A$ has ends in different $T_i$. Now, for each of the $t$ colors that appear in $A$, keep one edge of that color,  delete all other edges in $G$, and contract each $T_i$ to a single vertex to obtain $G'$. Since $G'$ has $t$ vertices with $t$ edges, there is a cycle $C$ in $G'$ of length at most $t$. Now, consider the edges of $C$ in the original graph $G$. Let the edges be $x_iy_{i+1}$ for $i \in \mathbb{Z} / t\mathbb{Z}$, such that $x_i$ and $y_i$ are contained in the same set $T_{j_i}$ in $\{T_1, \cdots, T_t\}$.

By~\cref{item3}, there is a rainbow path $P_i$ of $M$-colors between $x_i$ and $y_i$ contained in $T_{j_i}$ of length at most $|T_{j_i}|+d_{j_i}$. Joining the edges $x_i y_{i+1}$ with the paths $\{P_i\}$ yields a rainbow cycle of length at most
\begin{align*}
\frac{n+\sum_{c \in C_M(G)}(r-|c_G|)}{r} + t + \sum_{i=1}^t d_i &\leq
\frac{n+\sum_{c \in C(G)}(r-|c_G|)}{r}+f(r)+(4r+1)f(r) \\
&\leq \frac{n+\sum_{c \in C(G)}(r-|c_G|)}{r} + \alpha_r,
\end{align*}
as desired.

The remaining case is $V(G) \ne \bigcup_{i=1}^t T_i$. Let $W = V(G) \setminus \left(\bigcup_{i=1}^t T_i \right)$.  Let $D$ be the digraph with vertex set $W$ and for each $v \in W$ and each $M$-neighbor $u$ of $v$ in $W$, put an arc from $v$ to $u$ in $D$. By~\cref{item2}, every vertex $v \in W$ has at least one of its $M$-neighbors in $W$.  Therefore, $D$ has no sink.  By~\cref{item4}, at most $(4r+1)f(r)$ vertices in $W$ have an $M$-neighbor outside of $W$.  Therefore, $\df_r(D) \leq (r-1)(4r+1)f(r)$.
We now apply~\cref{cor:shen} in $D$ to obtain a directed cycle $C$ in $D$ of length at most
\begin{align*}
\frac{|W|+(r-1)(4r+1)f(r)}{r}+2r-1 &\leq \frac{n}{r} + \alpha_r.
\end{align*}
The corresponding edges of $C$ in $G$ form a rainbow cycle, which completes the proof of the lemma.
\end{proof}

\cref{lemma-manystars} and \cref{lemma-fewstars} together imply \cref{thm:main-gen} and \cref{thm:main}, as desired.

\subsection*{Acknowledgements} The second author is grateful to the  \href{https://www.matrix-inst.org.au/events/structural-graph-theory-downunder-ll/}{Structural Graph Theory Downunder II} workshop at the Mathematical Research Institute MATRIX (March 2022), for providing an ideal environment to work on this problem.   The second author also thanks Katie Clinch, Jackson Goerner, and Freddie Illingworth for very stimulating discussions.

  \let\oldthebibliography=\thebibliography
  \let\endoldthebibliography=\endthebibliography
  \renewenvironment{thebibliography}[1]{%
    \begin{oldthebibliography}{#1}%
      \setlength{\parskip}{0ex}%
      \setlength{\itemsep}{0ex}%
  }{\end{oldthebibliography}}

{\fontsize{11pt}{12pt}
\selectfont
\bibliographystyle{DavidNatbibStyle}
\bibliography{references}
}
\printindex
\end{document}